\documentclass[psamsfonts,12pt]{article}
\usepackage{amsmath}
\usepackage{amsthm}
\usepackage{amssymb}
\usepackage{amscd}
\usepackage{amsfonts}
\usepackage{amsbsy}
\usepackage{graphicx}
\usepackage[dvips]{psfrag}
\usepackage{setspace}
\usepackage{indentfirst}
\usepackage{rotating,graphics,psfrag,epsfig}
\usepackage{float}
\usepackage{hyperref}
\usepackage{booktabs}
\usepackage{xcolor}
\usepackage{mathrsfs}
\usepackage{overpic}
\usepackage{doi}
\usepackage{tikz}
\usepackage{lineno}

\newtheorem{theo}{Theorem}
\newtheorem{cor}{Corollary}

\newtheorem{lem}{Lemma}
\renewenvironment{proof}{\noindent {\bfseries Proof.}}{\hfill $\blacksquare$\vspace{0.4cm}}

\begin{document}


\begin{center}
{\Large{\bf Global stability of the Lengyel-Epstein systems}}
\end{center}

\bigskip

\begin{center}
{\large Lucas Queiroz Arakaki},\\
{\em Instituto de Matem\'atica, Estat\'istica e Computa\c c\~ao Cient\'ifica, \\ Universidade Estadual de Campinas,\\
Rua S\'ergio Buarque de Holanda, 651, CEP: 13.083-859, Campinas, SP, Brazil.} \\
e-mail: lucas.queiroz@unesp.br
\end{center}
\begin{center}
{\large Luis Fernando Mello$^\ast$ and Ronisio Moises Ribeiro},\\
{\em Instituto de Matem\'atica e Computa\c c\~ao, Universidade Federal de
Itajub\'a,\\ Avenida BPS 1303, Pinheirinho, CEP 37.500-903, Itajub\'a,
MG, Brazil.\\}
e-mail: lfmelo@unifei.edu.br, roniribeiro@unifei.edu.br
\end{center}

\bigskip

\begin{center}
{\bf Abstract}
\end{center}
We study the global (asymptotic) stability of the Lengyel-Epstein differential systems, sometimes called Belousov-Zhabotinsky differential systems. Such systems are topologically equivalent to a two-parameter family of cubic systems in the plane. We show that for each pair of admissible parameters the unique equilibrium point of the corresponding system is not globally (asymptotically) stable. On the other hand, we provide explicit conditions for this unique equilibrium point to be asymptotically stable and we study its basin of attraction. We also study the generic and degenerate Hopf bifurcations and highlight a subset of the set of admissible parameters for which the phase portraits of the systems have two limit cycles.

\bigskip



\noindent \makebox [40mm]{\hrulefill}

\noindent {\footnotesize $^\ast$Corresponding author}.


\section{Introduction and statement of the main results}\label{sec:1}
We study the Lengyel-Epstein differential systems \cite{VBFG}, sometimes called Belousov-Zhabotinsky differential systems \cite[page 229]{HSD},
\begin{equation}\label{eq:01}
\dot{x} = a - x - \frac{4 x y}{1 + x^2}, \quad \dot{y} = b x \left( 1 - \frac{y}{1 + x^2} \right),
\end{equation}
where  $x > 0$ and $y > 0$ are the state variables while $a > 0$ and $b > 0$ are parameters. The dot means derivative with respect to the independent variable $t$. See \cite{ES} and \cite{LRE}, in addition to the references cited above and those found there, to understand the chemical meaning of \eqref{eq:01}. Differential systems \eqref{eq:01} are topologically equivalent to
\begin{equation}\label{eq:02}
\dot{x} = (a - x) \left(1 + x^2 \right) - 4 x y, \quad \dot{y} = b x \left( 1 + x^2 - y \right).
\end{equation}
Here we will study systems \eqref{eq:02}, or equivalently the two-parameter family of cubic vector fields
\begin{equation}\label{eq:03}
F(x, y) = \Big( (a - x) \left(1 + x^2 \right) - 4 x y, b x \left( 1 + x^2 - y \right) \Big),
\end{equation}
keeping the parameters $a$ and $b$ positive, but allowing $x$ and $y$ to be real numbers. For the interpretations of our results with regard to the original differential systems \eqref{eq:01}, it is enough to consider the restriction to the set $x> 0$ and $y > 0$. Systems \eqref{eq:02}, or vector fields \eqref{eq:03}, have only one equilibrium point for all parameter values
$P_a = \left( a/5, 1 + a^2/25 \right)$.

A planar vector field is said to be bounded if the forward orbit of every initial condition enters and remains in a compact set. One of the most important problems in the qualitative theory of differential equations in the plane is the study of global (asymptotic) stability: an equilibrium point that is a global attractor. Our first main result goes in this direction.

\begin{theo}\label{thm:01}
The vector fields $F$ in \eqref{eq:03} are not bounded for all positive parameters $a$ and $b$.
\end{theo}

A direct consequence of Theorem \ref{thm:01} is following corollary.

\begin{cor}\label{cor:01}
The unique equilibrium point $P_a$ is not globally stable for all $a > 0$ and $b > 0$.
\end{cor}

In order to state our second main result, consider the subsets of admissible parameters
\begin{equation}\label{eq:sub01}
B = \left\{ (a, b): 0 < a \leq 5 \sqrt{\tfrac{5}{3}}, \,\, b > 0 \right\} \cup
\left\{ (a, b): a > 5 \sqrt{\tfrac{5}{3}}, \,\, b > b_H = \frac{3 a^2 - 125}{5 a} \right\},
\end{equation}
\begin{equation}\label{eq:05}
A = \left\{ (a, b): 0 < a \leq 3 \sqrt{3}, \,\, b > 0 \right\} \cup
\left\{ (a, b): a > 3 \sqrt{3}, \,\, b > b_a = a - 3 a^{1/3} \right\}.
\end{equation}
It follows that $A \subset B$. See Figure \ref{fig:2} (a).

\begin{theo}\label{thm:02}
If $(a, b) \in B$ then the equilibrium point $P_a$ is (locally) asymptotically stable. For each $(a, b) \in A$ the basin of attraction of the equilibrium point $P_a$ contains the half-plane $x > 0$.
\end{theo}

The proofs of Theorems \ref{thm:01} and \ref{thm:02} are presented in Section \ref{sec:2}. In Section \ref{sec:3} we study the generic and degenerate Hopf bifurcations of systems \eqref{eq:02}.

\section{Proofs of Theorems \ref{thm:01} and \ref{thm:02}}\label{sec:2}

\subsection{Proof of Theorem \ref{thm:01}}\label{subsec:1}

In order to prove Theorem \ref{thm:01} it is necessary to study the dynamics of systems \eqref{eq:02} at infinity via Poincar\'e compactification. See \cite[Chapter 5]{DLA}. Following the steps of \cite{DLA}, we study infinite equilibria using the local charts $U_1$ and $U_2$. Since systems \eqref{eq:02} are cubic, the study of infinite equilibria on the local charts $V_1$ and $V_2$ follows directly from the study on the local charts $U_1$ and $U_2$. In the local $U_1$, systems \eqref{eq:02} are written as
\begin{equation}\label{u2eq:00}
\dot{u} =b + u - (a + b - 4 u) u v + (b + u) v^2 - a u v^3, \quad \dot{v} = v (1 + 4 u v + v^2 - a (v + v^3)).
\end{equation}
Setting $v = 0$, the only infinite equilibrium point covered by this chart is $I_1 = (-b,0)$, which is an unstable node since the eigenvalues of the Jacobian matrix of the vector field associated to \eqref{u2eq:00} are both equal to 1. The same occurs with the infinite equilibrium point $I_2$ in the local chart $V_1$.

In the local chart $U_2$, systems \eqref{eq:02} write as
\begin{equation}\label{u2eq:01}
	\dot{u} =-b{u}^{4}-b{u}^{2}{v}^{2}+a{u}^{2}v+a{v}^{3}+b{u}^{2}v-{u}^{3}-u{v}^{2
}-4\,uv
, \quad \dot{v} = -b{u}^{3}v-bu{v}^{3}+bu{v}^{2}.
\end{equation}
The only equilibrium point of the local chart $U_2$ which is not covered by the local chart $U_1$ is $I_3 = (0, 0)$. Since it is a degenerate equilibrium point, we use the quasi-homogeneous polar blow-ups to describe the local dynamics at this point. See \cite{AFJ} and \cite{DLA}.

Consider the blow-up given by $(u, v)=(r \cos(\theta), r^2 \sin(\theta))$, with $r \geq 0$ and $0 \leq \theta <2\pi$. It transforms systems \eqref{u2eq:01}, after eliminating the common factor $(1+\sin^2(\theta))/r^2$, into
\begin{equation}\label{eq:bup}
\begin{split}
&\dot{r}=r \bigg( -\frac{3}{4} \cos^2(\theta) + \cos(\theta) \bigg( -\frac{1}{4} \cos(3\theta) - 2 \sin(2\theta) \bigg) \bigg) + r^2 \cos(\theta) \bigg( -\frac{b}{2} - \frac{b}{2} \cos(2\theta)\\
& \quad\,\, + \bigg(\frac{a}{4} + b \bigg) \sin(\theta) + \frac{a}{4} \sin(3\theta) \bigg)
+ r^3 \bigg( -\frac{1}{4} \cos^2(\theta) + \frac{1}{4} \cos(\theta) \cos(3\theta) \bigg)
+ r^4 \cos(\theta) \\
&\quad\,\,\, \bigg( -\frac{b}{2} + \frac{b}{2} \cos(2\theta) + \frac{3a}{4} \sin(\theta) - \frac{a}{4} \sin(3\theta) \bigg), \\
&\dot{\theta}= 2 \cos(\theta) \sin(\theta) \left( \cos^2(\theta) + 4 \sin(\theta) \right)
+ r \bigg( b \cos^4(\theta) \sin(\theta) + (-2a - b) \cos^2(\theta) \sin^2(\theta) \bigg)\\
&\quad\,\, + r^2 \sin^2(\theta) \sin(2\theta)
+ r^3 \bigg( b \cos^2(\theta) \sin^3(\theta) - 2a \sin^4(\theta) \bigg).
\end{split}
\end{equation}
The equilibrium point $I_3 = (0,0)$ is blown up into the circle $r=0$. The equilibrium points belonging to this circle are $\theta^* \in \left\{ 0, \pi/2, \pi, \pi+\arcsin{(\sqrt{5}-2)}, 3\pi/2, 2\pi-\arcsin{(\sqrt{5}-2)} \right\}$, solutions of the equation $2 \cos(\theta) \sin(\theta) \left( \cos^2(\theta) + 4 \sin(\theta) \right)=0$. Denoting by $J(r,\theta)$ the Jacobian matrix of the vector fields associated to \eqref{eq:bup}, we have that
\begin{eqnarray*}
J(0,0)=J(0,\pi)=\left(\begin{array}{cc}
    -1 & 0 \\
    0 & 2
\end{array}\right),\; J(0,\pi/2)=-J(0,3\pi/2)=\left(\begin{array}{cc}
    0 & 0 \\
    0 & -8
\end{array}\right),\\
J(0,\pi+\theta_0)=J(0,2\pi-\theta_0)=\left(\begin{array}{cc}
    0 & 0 \\
    (152-68\sqrt{5})(2a+5b) & 320-144\sqrt{5}
\end{array}\right),
\end{eqnarray*}
where $\theta_0=\arcsin{(\sqrt{5}-2)}$. From the Jacobian matrices alone, we can conclude that $(0,0)$ and $(0,\pi)$ are hyperbolic saddles, attracting in the radial direction ($r$-direction) and repelling in the angular direction ($\theta$-direction). The other equilibrium points are semi-hyperbolic and require a more detailed analysis.

The semi-hyperbolic equilibrium points can be classified using classical theorems from the literature, for instance \cite[Theorem 2.19]{DLA}. After a translation and a linear change of variables, each semi-hyperbolic equilibrium point can be studied as the origin of systems of the form
\begin{equation*}\label{eq:semihyp}
\dot{r}=R(r,\theta),\quad \dot{\theta}=\lambda\theta+\Theta(r,\theta),
\end{equation*}
with $\lambda\neq 0$ and $\Theta$ having neither constant nor linear terms in its power series expansion. Define the function $\theta=f(r)$ as the solution of $\dot{\theta}=0$ and consider the power series expansion of the function $R(r,f(r))=a_mr^m+O(r^{m+1})$, with $a_m\neq 0$ to determine the behavior of the equilibrium point. Performing this investigation, we obtain that:
\begin{itemize}
\item For $(0,\pi/2)$, it follows that $\lambda=-8$, $a_m = ab/4 > 0$ and $m=5$, which implies that this equilibrium point is a topological saddle repelling in the radial direction and attracting in the angular direction;

\item For $(0,3\pi/2)$, it follows that $\lambda=8$, $a_m = -ab/4 < 0$ and $m=5$, and thus this equilibrium is a topological saddle attracting in the radial direction and repelling in the angular direction;

\item For $(0,\pi+\theta_0)$, it follows that $\lambda=144\sqrt{5}-320>0$, $a_m>0$ and $m=2$. In this case, the equilibrium point is a saddle-node with attractor parabolic sector in the half-plane $r < 0$ and the two hyperbolic sectors in the half-plane $r > 0$;

\item For $(0,2\pi-\theta_0)$, it follows that $\lambda=144\sqrt{5}-320>0$, $a_m<0$ and $m=2$. In this case, the equilibrium is also a saddle-node but with attractor parabolic sector in the half-plane $r > 0$ and the two hyperbolic sectors in the half-plane $r < 0$.

\end{itemize}

\begin{figure}[h]
\begin{center}
\begin{overpic}[width=5.5in]{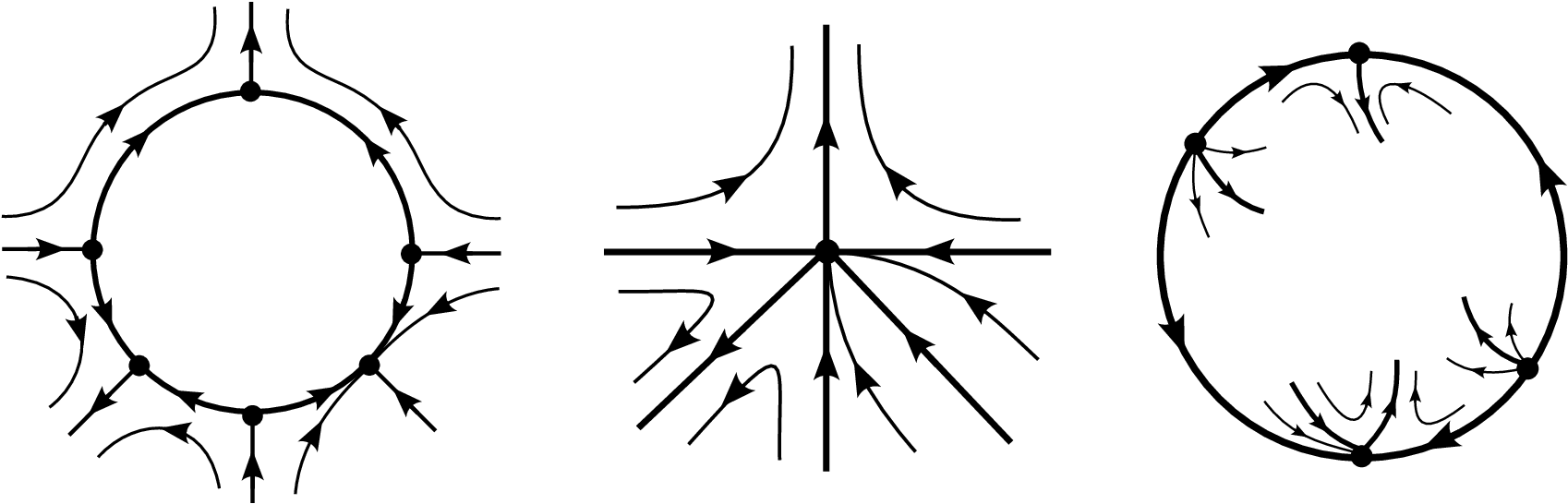}
\put(33,15) {$\theta_1$}
\put(15,33) {$\theta_2$}
\put(-3,15) {$\theta_3$}
\put(2,2) {$\theta_4$}
\put(15,-3) {$\theta_5$}
\put(28,2) {$\theta_6$}
\put(67,14) {$u$}
\put(53.5,30) {$v$}
\put(14,-7) {$(a)$}
\put(51,-7) {$(b)$}
\put(85,-7) {$(c)$}
\put(99,7) {$I_1$}
\put(72.5,23) {$I_2$}
\put(86,30.5) {$I_3$}
\put(86,-0.5) {$I_4$}
\end{overpic}
\end{center} \vspace{0.5cm}
\caption{(a) Desingularization of systems \eqref{u2eq:01} using polar blow-ups where $\theta_1=0$, $\theta_2=\pi/2$, $\theta_3=\pi$, $\theta_4=\pi+\arcsin(\sqrt{5} - 2)$, $\theta_5=3\pi/2$ and $\theta_6=2\pi - \arcsin(\sqrt{5} - 2)$. (b) Topological local phase portraits at the origin of systems \eqref{u2eq:01}. (c) Dynamics near infinity of systems \eqref{eq:02} in the Poincar\'e disk. The equilibrium point $I_4$ has a parabolic sector that intersects the interior of the Poincar\'e disk.}
\label{fig:1}
\end{figure}
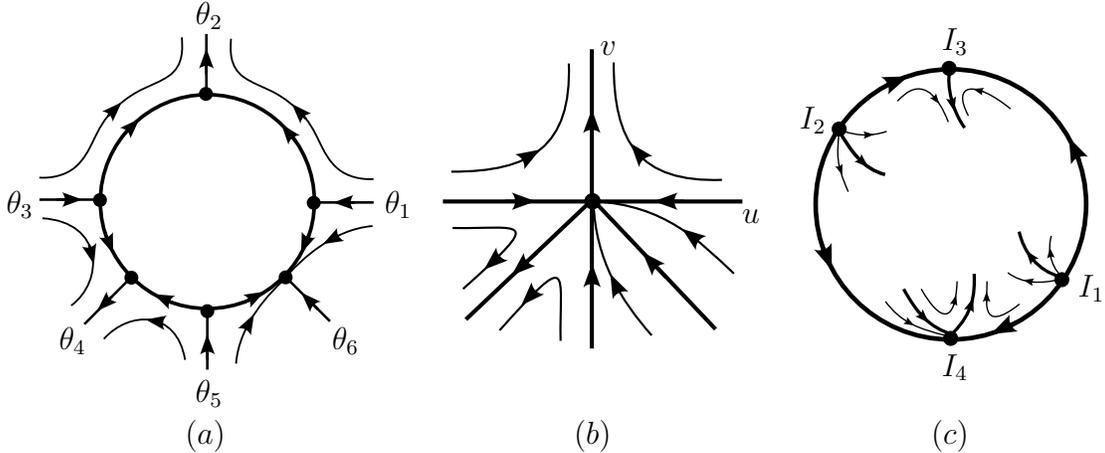

A summary of the study of dynamics at $r=0$ is illustrated in Figure \ref{fig:1} (a). Going back through the changes of variables, the dynamics near the circle $r=0$ allow us to classify the infinite equilibrium point at the origin of systems \eqref{u2eq:01}. See Figure \ref{fig:1} (b). It follows that the equilibrium has four hyperbolic sectors and one parabolic sector. The local behaviors of the infinite points studied in this subsection are depicted in Figure \ref{fig:1} (c).

\medskip

\noindent\textbf{Proof of Theorem \ref{thm:01}.} 
Theorem \ref{thm:01} is a consequence of the analysis of the infinite equilibrium points $I_1$, $I_2$, $I_3$, and $I_4$. We emphasize that the equilibrium point $I_4$ has a parabolic sector that intersects the interior of the Poincar\'e disk. See Figure \ref{fig:1} (c). It follows that there are initial conditions $X_0 = (x_0, y_0) \in \mathbb R^2$ such that their orbits $\gamma_{X_0}$ with maximal interval of definition $(t_{X_0}^{-}, t_{X_0}^{+})$ satisfy $\gamma_{X_0} \to I_4$ as $t \to t_{X_0}^{+}$. \hfill $\blacksquare$\vspace{0.4cm}

\subsection{Proof of Theorem \ref{thm:02}}\label{subsec:2}

The following lemma is an adaptation of results proved in \cite{VBFG}. The first quadrant of the plane is defined usually as $Q_1 = \left\{ (x, y) \in \mathbb R^2: x > 0, \,\, y > 0 \right\}$.

\begin{lem}\label{lemma:sub03}
If $(a, b) \in A$ given in \eqref{eq:05} then systems \eqref{eq:02} do not have closed orbits in $Q_1$.
\end{lem}

\begin{proof}
Define the positive function $h: Q_1 \longrightarrow \mathbb R$, $h(x, y) = 1/x$, and the vector fields $G(x,y) = h(x, y) F(x, y)$. It follows that the divergence of $G$ is given by $f(x) = a - b - 2 x - a/x^2$. By the Bendixson-Dulac Theorem (see \cite[page 189]{DLA}) the lemma will be proven if we show that $f(x) < 0$, for all $x > 0$ and $(a, b) \in A$. If $b > b_a$ then $x^2 f(x) < g(x) = - 2 x^3 + 3 a^{1/3} x^2 - a$, for all $x > 0$ and $(a, b) \in A$. It is enough to prove that $g(x) \leq 0$, for all $x > 0$ and $a > 0$. The critical points of $g$ are $x = 0$ (local minimum) and $x = a^{1/3}$ (local maximum) with $g(0) = -a < 0$ and $g(a^{1/3}) = 0$. So, it follows that $g(x) \leq 0$ for all $x > 0$ and $(a, b) \in A$.
\end{proof}

\begin{lem}\label{lemma:sub04}
The sets $Q_1$ and the half-plane $x > 0$ are positively invariant by the flows of \eqref{eq:02}.
\end{lem}

\begin{proof}
Evaluate the vector fields $F$ in \eqref{eq:03} on the boundary of $Q_1$. On the one hand, $F(0, y) = (a, 0)$ points east for all $y \in \mathbb R$. On the other hand, the second component of $F(x,0)$ is $b x (1 + x^2) > 0$, for all $x > 0$. So $F(x,0)$ points northeast, north or northwest. So, $Q_1$ is positively invariant by the flows of systems \eqref{eq:02}. As $F(0, y) = (a, 0)$ points east for all $y \in \mathbb R$, it follows that the half-plane $x > 0$ is also positively invariant by the flows of systems \eqref{eq:02}.
\end{proof}

\noindent\textbf{Proof of Theorem \ref{thm:02}.} By linear analysis, it follows that the equilibrium point $P_a \in Q_1$ is (locally) asymptotically stable if $(a, b) \in B$ (see equation \eqref{eq:sub01}). Arbitrarily fix a vector field $F$ in \eqref{eq:03} with $(a, b) \in A$. By Lemma \ref{lemma:sub04}, the half-plane $x > 0$ is positively invariant by the flow of $F$. Take an initial condition $X_1 = (x_1, y_1)$ in the half-plane $x > 0$ and consider its positive orbit $\gamma_{X_1}^{+}$ defined in the maximal interval $[0, t_{X_1}^{+})$. By the analysis performed in the proof of Theorem \ref{thm:01}, the equilibrium points at infinite are not the $\omega$-limit set of $\gamma_{X_1}^{+}$. So, there is a compact set $K$ contained in the half-plane $x > 0$ such that $\gamma_{X_1}^{+}$ enters and remains in $K$. This implies that $t_{X_1}^{+} = \infty$. Thus the $\omega$-limit set of $\gamma_{X_1}^{+}$ is nonempty and is also contained in $K$. The unique equilibrium point of $F$ is $P_a \in Q_1$ which is (locally) asymptotically stable. By Lemma \ref{lemma:sub04}, $Q_1$ is positively invariant by the flow of $F$ and, by Lemma \ref{lemma:sub03}, $F$ does not have closed orbits in $Q_1$. Putting all these statements together, and by using the Poincar\'e-Bendixson Theorem, the $\omega$-limit set of $\gamma_{X_1}^{+}$ is $\{P_a\}$. In short, Theorem \ref{thm:02} is proved.\hfill $\blacksquare$\vspace{0.4cm}

\section{Hopf bifurcations and limit cycles}\label{sec:3}

Define the curve $H = \{ (a, b): b = b_H, \, a > 5 \sqrt{5/3} \}$ (see equation \eqref{eq:sub01} and Figure \ref{fig:2} (a)). From the first statement of Theorem \ref{thm:02} and linear analysis it follows that $P_a$ is hyperbolic and unstable for each pair of admissible parameters on the complement of the set $B \cup H$. For $(a, b) \in H$, $P_a$ is a Hopf point. Following the approach given in \cite[page 292]{HW} for the calculation of the first Lyapunov coefficient $L_1$, it follows that $L_1 (a) = m(a) \, (2 a^4-675 a^2-3125)$, where $m = m(a)$ is a positive function of the parameter $a$.

Consider $a_B = 5 \sqrt{27+\sqrt{769}}/2$, $B_a = \left( a_B, b_H(a_B) \right)$. Thus, $L_1(a) < 0$ for $a \in \left( 5 \sqrt{5/3}, a_B \right)$, $L_1(a_B) = 0$, and $L_1 (a) > 0$ for $a > a_B$. Computing the second Lyapunov coefficient, we obtain $L_2 (a_B) < 0$. The point $B_a \in H$ corresponds to a Bautin point associated to a degenerate Hopf bifurcation. Based on the above statements, we conclude: $P_a$ is asymptotically stable for $(a,b)$ on the arc $H_{-} \subset H$ corresponding to $a \in \left( 5 \sqrt{5/3}, a_B \right]$ and it is unstable for $(a,b)$ on the arc $H_{+} \subset H$ corresponding to $a > a_B$. See Figure \ref{fig:2} (a). From the classical Hopf Bifurcation Theorem, the phase portrait of \eqref{eq:02} has a stable limit cycle surrounding the unstable equilibrium $P_a$ if we take parameters close to but below the arc $H_{-}\setminus \{B_a\}$. On the other hand, the phase portrait of \eqref{eq:02} has an unstable limit cycle surrounding the stable equilibrium $P_a$, if we take parameters close to but above the arc $H_{+}$.

\begin{figure}[H]
\centerline{
\fbox{\includegraphics[width=16.5cm]{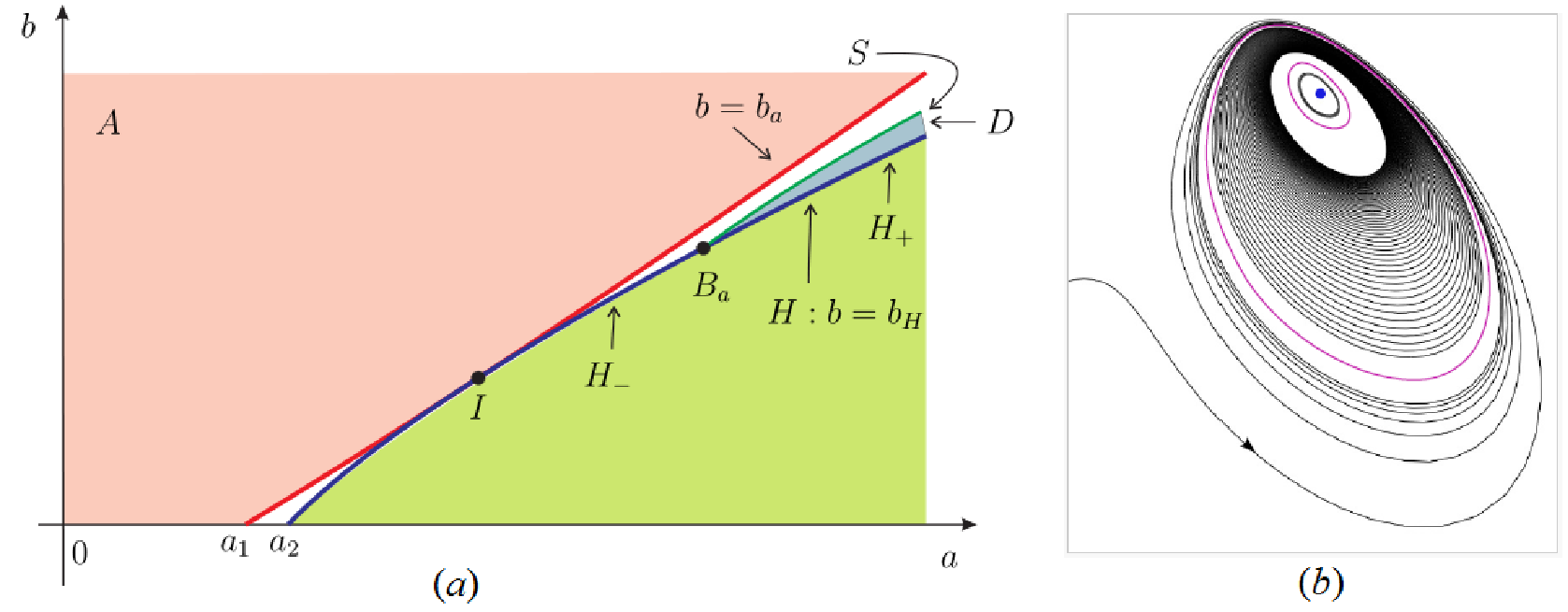}}}
\caption{(a) The set of admissible parameters: $A$ is defined in \eqref{eq:05}, $H$ is the Hopf curve, $a_1 = 3 \sqrt{3}$, $a_2 = 5 \sqrt{5/3}$, $I = (5 \sqrt{5}, 2 \sqrt{5})$ is the intersection point of the graphs $b = b_H$ (blue) and $b = b_a$ (red). (b) Phase portrait of \eqref{eq:02} for $(a, b) = (24.712, 13.85) \in D$.}
\label{fig:2}
\end{figure}

In the set of admissible parameters there is a curve $S$ with endpoint at $B_a$ corresponding to parameters for which the phase portrait of \eqref{eq:02} has a semistable limit cycle surrounding the stable equilibrium $P_a$. There is also a region $D$ bounded by the curves $S$ and $H_{+}$ corresponding to parameters for which the phase portrait of \eqref{eq:02} has two nested limit cycles surrounding the stable equilibrium $P_a$: the smaller one that is unstable and the larger one that is stable. See Figure \ref{fig:2} (a) for an illustration of $D$ and Figure \ref{fig:2} (b) for an illustration of the phase portrait of \eqref{eq:02} with two limit cycles.

\section*{Acknowledgments}

\noindent The first author is supported by Funda\c c\~ao de Amparo \`a Pesquisa do Estado de S\~ao Paulo (grant number 2024/06926-7). The second is partially supported by Funda\c c\~ao de Amparo \`a Pesquisa do Estado de Minas Gerais (grant numbers APQ-02153-23 and RED-00133-21), by Funda\c c\~ao de Amparo \`a Pesquisa do Estado de S\~ao Paulo (grant number 2019/07316–0) and by Conselho Nacional de Desenvolvimento Cient\'ifico e Tecnol\'ogico (grant number 311921/2020-5). The third author is grateful for the warm hospitality of the Universidade Federal de Itajub\'a during the development of this work.

\section*{Conflict of interest}

\noindent The authors declare that they have no conflict of interest.

\section*{Data availability}

\noindent Data sharing is not applicable to this article as no new data were created or analyzed in this study.

\end{document}